\documentclass[amsmath,amssymb,aps,floatfix,pre,showpacs,superscriptaddress,preprint]{revtex4}
\usepackage{amsthm}
\newtheorem{theorem}{Theorem}[section]
\newtheorem{lemma}[theorem]{Lemma}

\newtheorem{corollary}[theorem]{Corollary}
\newtheorem{definition}{Definition}[section]
\newtheorem{remark}{Remark}[section]
\newtheorem{example}{Example}[section]
\begin{document}

\title{A class of CTRWs: \\ Compound fractional Poisson processes}

\author{Enrico Scalas
}

\affiliation{Dipartimento di Scienze e Tecnologie Avanzate, \\ Universit\`a del Piemonte Orientale,\\
viale T. Michel 11, 15121 Alessandria, Italy \\
enrico.scalas@mfn.unipmn.it
}

\begin{abstract}
This chapter is an attempt to present a mathematical theory of compound fractional Poisson processes. 
The chapter begins with the characterization of a well-known L\'evy process: The compound Poisson process.
The semi-Markov extension of the compound Poisson process naturally leads to the compound fractional Poisson process, where the
Poisson counting process is replaced by the Mittag-Leffler counting process also known as fractional Poisson process. 
This process is no longer Markovian and L\'evy. 
However, several analytical results are available and some of them are discussed here.
The functional limit of the compound Poisson process is an $\alpha$-stable L\'evy process, whereas in the case of the
compound fractional Poisson process, one gets an $\alpha$-stable L\'evy process subordinated to the fractional Poisson process.
\end{abstract}

\maketitle


\section{Introductory notes}

This chapter is an attempt to present a mathematical theory of compound fractional Poisson processes. It is
not completely self-contained. The proofs of some statements can be found in widely available textbooks or papers. In several cases,
freely downloadable versions of these papers can be easily retrieved.

The chapter begins with the characterization of a well-known L\'evy process: The compound Poisson process. This process is
extensively discussed in the classical books by Feller \cite{feller} and de Finetti \cite{definetti}.

The semi-Markov extension of the compound Poisson process naturally leads to the compound fractional Poisson process, where the
Poisson counting process is replaced by the Mittag-Leffler counting process also called fractional Poisson process 
\cite{hilfer,hilferanton}. 
This process is no longer Markovian and L\'evy. 
However, several analytical results are available and some of them are discussed below.

The functional limit of the compound Poisson process is an $\alpha$-stable L\'evy process, whereas in the case of the
compound fractional Poisson process, one gets an $\alpha$-stable L\'evy process subordinated to the fractional Poisson process.

I became interested in these processes as possible models for tick-by-tick financial data. The main results obtained by my co-workers
and myself are described in a review paper for physicists \cite{scalas06}.

The reader interested in Monte Carlo simulations can consult two recent papers \cite{fulger08,germano09} where algorithms are presented to
simulate the fractional compound Poisson process.

\section{Compound Poisson process and generalizations}\label{cpp}

Let $\{X_i \}_{i=1}^\infty$ be a sequence of independent and identically distributed (i.i.d.) real-valued random variables with
cumulative distribution function $F_X (x)$, and let
$N(t)$, $t \geq 0$ denote the Poisson process. Further assume that the i.i.d. sequence and the Poisson process are independent.
We have the following
\begin{definition}[Compound Poisson process]
The stochastic process
\begin{equation}
\label{compoundpoisson}
Y(t) = \sum_{i=1}^{N(t)} X_i
\end{equation}
is called compound Poisson process.
\end{definition}
Here, we shall consider the one-dimensional case only. The extension of many results to $\mathbb{R}^d$ is often 
straightforward. The compound Poisson process can be seen as a random walk subordinated to a Poisson process; in other words, it is
a random sum of independent and identically distributed random variables. It turns out that the compound Poisson
process is a L\'evy process.
\begin{definition}[L\'evy process]
A stochastic process $Y(t)$, $t \geq 0$ with $Y(0) = 0$ is a L\'evy process if the following three conditions
are fulfilled
\begin{enumerate}
\item (independent increments) if $t_1 < t_2 < \cdots < t_n$, the increments $Y(t_2) - Y(t_1), \ldots, 
Y(t_n) - Y(t_{n-1})$ are independent random variables;
\item (time-homogeneous increments) the law of the increment $Y(t + \Delta t) - Y(t)$ does not depend on $t$;
\item (stochastic continuity) $\forall a > 0$, one has that $\lim_{\Delta t \to 0} \mathbb{P} (|Y(t + \Delta t) - Y(t)| \geq a) = 0$.
\end{enumerate}
\end{definition}
Loosely speaking, one can say that L\'evy processes are stochastic processes with stationary and independent increments.
Due to Kolmogorov's extension theorem \cite{billingsley}, a stochastic process is characterized by its
finite dimensional distributions. In the case of a L\'evy process, the knowledge of the law of $Y(\Delta t)$ 
is sufficient to compute any finite dimensional distribution. Let us denote by $f_{Y(\Delta t)}(y,\Delta t)$ the probability density function
of $Y(\Delta t)$ 
\begin{equation}
\label{1pfidi}
f_{Y(\Delta t)}(y,\Delta t) \, dy \stackrel{\text{def}}{=} \mathbb{P}(Y(\Delta t) \in dy).
\end{equation}
As an example, suppose you want to know the joint density function $f_{Y(t_1),Y(t_2)}(y_1,t_1;y_2,t_2)$ defined as
\begin{equation}
\label{2pfidi}
f_{Y(t_1),Y(t_2)}(y_1,t_1;y_2,t_2)\, dy_1 \, dy_2 \stackrel{\text{def}}{=}  \mathbb{P}(Y(t_1) \in d y_1, Y(t_2) \in d y_2). 
\end{equation}
This is given by
\begin{multline}
\label{fidi}
f_{Y(t_1),Y(t_2)}(y_1,t_1;y_2,t_2) dy_1 dy_2  \stackrel{\text{def}}{=}  \mathbb{P}(Y(t_1) \in d y_1, Y(t_2) \in d y_2)  \\ 
 =  \mathbb{P}(Y(t_2) \in dy_2|Y(t_1) \in dy_1) \mathbb{P}(Y(t_1) \in dy_1) \\
 =  \mathbb{P}(Y(t_2) - Y(t_1) \in d(y_2 - y_1)) \mathbb{P}(Y(t_1) \in d y_1) \\
 =  f_{Y(t_2) - Y(t_1)}(y_2 - y_1, t_2 - t_1) f_{Y(t_1)}(y_1,t_1) dy_1 dy_2,
\end{multline}
and this procedure can be used for any finite dimensional distribution. The extension theorem shows the existence
of a stochastic process given a suitable set of finite dimensional distributions obeying Komogorov's consistency
conditions \cite{billingsley}, but not the uniqueness.
\begin{definition}[c\`adl\`ag process]
A stochastic process $Y(t)$, $t \geq 0$ is c\`adl\`ag (continu \`a droite et limite \`a gauche) if its realizations are right-continuous with left limits.
\end{definition}
A c\`adl\`ag stochastic process has realizations with jumps. Let $\bar{t}$ denote the epoch of a jump. Then, in a c\`adl\`ag process, one has
$Y(\bar{t}) = Y(t^+) \stackrel{\text{def}}{=} \lim_{t \to \bar{t}^+} Y(t)$.
\begin{definition}[Modification of a process]
A modification $Z(t)$, $t \geq 0$, of a stochastic process $Y(t)$, $t \geq 0$, is a stochastic process on the same probability space
such that $\mathbb{P}(Z(t) = Y(t)) = 1$.
\end{definition}
\begin{theorem}
Every L\'evy process has a unique c\`adl\`ag modification.
\end{theorem}
\begin{proof}
For a proof of this theorem one can see the first chapter of the book by Sato \cite{sato}.
\end{proof}
The following theorem gives a nice characterization of compound Poisson processes.
\begin{theorem}
Y(t) is a compound Poisson process if and only if it is a L\`evy process and its realizations are piecewise
constant c\`adl\`ag functions.
\end{theorem}
\begin{proof}
An accessible proof of this theorem can be found in the book by Cont and Tankov \cite{cont}.
\end{proof}
As a consequence of the above results, the compound Poisson process enjoys all the properties of L\'evy processes,
including the Markov property. To show that a L\'evy process has the Markov property, some further definitions
are necessary.
\begin{definition}[Filtration]
A family $\mathcal{F}_t$, $t \geq 0$ of $\sigma$-algebras is a filtration if it is non-decreasing, meaning that
$\mathcal{F}_s \subseteq \mathcal{F}_t$ for $0 \leq s \leq t$.
\end{definition}
\begin{definition}[Adapted process]
A process $Y(t)$, $t \geq 0$ is said adapted to the filtration $\mathcal{F}_t$, $t \geq 0$ if it is $\mathcal{F}_t$-measurable
for $t \geq 0$.
\end{definition}
\begin{definition}[Markov process with respect to a filtration]
A process $Y(t)$ is a Markov process with respect to the filtration $\mathcal{F}_t$, $t \geq 0$ if it is
adapted to $\mathcal{F}_t$, $t \geq 0$ and ($A \subset \mathbb{R}$)
\begin{equation}
\label{markovproperty}
\mathbb{P}(Y(t) \in A|\mathcal{F}_s) = \mathbb{P}(Y(t) \in A|Y(s)).
\end{equation}
\end{definition}
\begin{definition}[Natural filtration]
The natural filtration for a stochastic process is the family of non-decreasing $\sigma$-algebras generated by
the process itself $\{ \sigma(X(s)), \, s \in [0,t] \}$, $t \geq 0$.
\end{definition}
\begin{definition}[Markov process with respect to itself]
A process $Y(t)$, $t \geq 0$ is a Markov process with respect to itself (or simply a Markov process) if it is a
Markov process with respect to its natural filtration.
\end{definition}
The natural filtration $\{ \sigma(X(s)), \, s \in [0,t] \}$, $t \geq 0$, is a formal way to characterize the history of the process up to time $t$. 
For a Markov process, the future values do not depend on the whole history, but only on the present value of the process.
\begin{definition}[Transition probability]
Given the Markov process $Y(t)$, $t \geq 0$, its transition probability $P(y, A,\Delta t, t)$ is defined as
\begin{equation}
P(y, A,\Delta t, t) = \mathbb{P}(Y(t + \Delta t) \in A|Y(t)=y),
\end{equation}
where $A \subset \mathbb{R}$.
\end{definition}
\begin{definition}[Homogeneous Markov process]
A Markov process $Y(t)$, $t \geq 0$ is said (time)-homogeneous if its transition probability $P(y, A,\Delta t, t)$ does not depend on $t$.
\end{definition}
\begin{theorem}
A L\'evy process is a time-homogeneous Markov process with transition probability
\begin{equation}
\label{markovtrans}
P(y, A, \Delta t) = \mathbb{P}(Y(\Delta t) \in A - y) = \int_{x \in A} f_{Y(\Delta t)}(x-y, \Delta t) \, dx.
\end{equation}
\end{theorem}
\begin{proof}
The Markov property is a consequence of the independence of increments. The following chain of equalities holds true
\begin{eqnarray}
\label{levymarkov1}
\mathbb{P}(Y(t+\Delta t) \in A| \mathcal{F}_t) & = & \mathbb{P}(Y(t+\Delta t) - Y(t) \in A - Y(t)| \mathcal{F}_t) \nonumber \\
& = & \mathbb{P}(Y(t+\Delta t) - Y(t) \in A - Y(t)| Y(t)).
\end{eqnarray}
We further have
\begin{equation}
\label{levymarkov2}
\mathbb{P}(Y(t + \Delta t) \in A|Y(t) = y) = \mathbb{P}(Y(\Delta t) \in A-y) = \int_{x \in A} f_{Y(\Delta t)}(x-y, \Delta t) \, dx,
\end{equation}
as a consequence of time homogeneity.
\end{proof}
This result fully justifies the passages leading to equation (\ref{fidi}).

It turns out that $f_{Y(\Delta t)}(y, \Delta t)$ can be explicitly written for a compound Poisson process. Let
$F_X (x)$ be the law of the jumps $\{X_i\}$ and let $\lambda$ denote the parameter of the Poisson
process, then we have the following
\begin{theorem}
\label{fundtheor}
The cumulative distribution function of a compound Poisson process is given by
\begin{equation}
\label{CTRWcumulative}
F_{Y(t)} (y,t) = \mathrm{e}^{-\lambda t} \sum_{n=0}^\infty \frac{(\lambda t)^n}{n!} F^{*n}_{Y_n} (y),
\end{equation}
where $F^{*n}_{Y_n} (y)$ is the $n$-fold convolution of $F_X (x)$.
\end{theorem}
\begin{proof}
Starting from $Y(0) = 0$, at time $t$, there have been $N(t)$ jumps, with $N(t)$ assuming integer values
starting from $0$ ($N(t) = 0$ means no jumps up to time $t$). To fix the ideas, suppose that $N(t) = n$. Therefore,
one has
\begin{equation}
Y(t) = \sum_{i=1}^{N(t)} X_i = \sum_{i=1}^n X_i = Y_n
\end{equation}
and, in this case,
\begin{equation}
F_{Y_n} (y) = \mathbb{P}(Y_n \leq y) = \mathbb{P} \left( \sum_{i=1}^n X_i \leq y \right) = F^{*n}_{Y_n} (y).
\end{equation}
For the Poisson process, one has
\begin{equation}
\label{poissonprob}
P(n,t) \stackrel{\text{def}}{=}  \mathbb{P}(N(t) = n) = \mathrm{e}^{-\lambda t} \frac{(\lambda t)^n}{n!}.
\end{equation}
Given the independence between $N(t)$ and the $X_i$s, one has that
\begin{equation}
\label{jointjump}
\mathbb{P}(Y_n \leq y, N(t)=n) = P(n,t) F^{*n}_{Y_n} (y) = \mathrm{e}^{-\lambda t} \frac{(\lambda t)^n}{n!} F^{*n}_{Y_n} (y).
\end{equation}
The events $\{Y_n \leq y, N(t)=n\}$ are mutually exclusive and exhaustive, meaning that
\begin{equation}
\label{eventunion}
\{ Y(t) \leq y \} = \cup_{n=0}^{\infty} \{Y_n \leq y, N(t)=n\},
\end{equation}
and that, for any $m \neq n$
\begin{equation}
\label{eventintersection}
\{Y_m \leq y, N(t)=m\} \cap \{Y_n \leq y, N(t)=n\} = \varnothing.
\end{equation}
Calculating the probability of the two sides in equation (\ref{eventunion}) and using equation (\ref{jointjump}) and the
axiom of infinite additivity
\begin{multline}
\label{infiniteadditivity}
F_{Y(t)} (y,t) = \mathbb{P}(Y(t) \leq y ) = \mathbb{P}\left(\cup_{n=0}^{\infty} \{Y_n \leq y, N(t)=n\}\right) \\
= \sum_{n=0}^{\infty} \mathbb{P}(Y_n \leq y, N(t)=n) = \sum_{n=0}^{\infty} P(n,t) F^{*n}_{Y_n} (y) \\ 
= \mathrm{e}^{-\lambda t} \sum_{n=0}^\infty \frac{(\lambda t)^n}{n!} F^{*n}_{Y_n} (y)
\end{multline}
leads to the thesis.
\end{proof}
\begin{remark}[Generality of Theorem \ref{fundtheor}]
{\em Note that the above theorem is valid for any counting process $N(t)$ in the following form}
\end{remark}
\begin{theorem}
\label{fundtheorgen}
Let $\{X \}_{i=1}^\infty$ be a sequence of i.i.d. real-valued random variables with cumulative distribution function $F_X(x)$ and let
$N(t)$, $t \geq 0$ denote a counting process independent of the previous sequence and 
such that the number of events in the interval $[0,t]$ is a
finite but arbitrary integer $n = 0, 1, \ldots$. Let $Y(t)$ denote the process
\begin{equation}
\label{compoundrenewal}
Y(t) = \sum_{i=1}^{N(t)} X_i.
\end{equation}
Then if $P(n,t) \stackrel{\text{def}}{=}  \mathbb{P}(N(t) = n)$, the cumulative distribution function of $Y(t)$ is given by
\begin{equation}
\label{cumulativecompoundrenewal}
F_{Y(t)} (y,t) = \sum_{n=0}^{\infty} P(n,t) F^{*n}_{Y_n} (y),
\end{equation}
where $F^{*n}_{Y_n} (y)$ is the $n$-fold convolution of $F_X (x)$.
\end{theorem}
\begin{proof}
The proof of this theorem runs exactly as the proof of theorem \ref{fundtheor} without specifying $P(n,t)$.
\end{proof}
Theorem \ref{fundtheorgen} will be useful in the next section when the Poisson process will be replaced by the
fractional Poisson process.
\begin{remark}[The $n=0$ term]
{\em For $n=0$, one assumes $F^{*0}_{Y_0} (y) = \theta (y)$ where $\theta(y)$ is the Heaviside function. Note that $P(0,t)$ is nothing else
but the survival function at $y=0$ of the counting process. Therefore, equation (\ref{cumulativecompoundrenewal}) can be 
equivalently written as}
\begin{equation}
\label{alternatecum}
F_{Y(t)} (y,t) = P(n,0) \, \theta(y) + \sum_{n=1}^{\infty} P(n,t) F^{*n}_{Y_n} (y),
\end{equation}
\end{remark}
\begin{remark}[Uniform convergence]
{\em The series (\ref{CTRWcumulative}) and (\ref{cumulativecompoundrenewal}) are uniformly convergent for $y \neq 0$ and 
for any value of $t \in (0,\infty)$ (this statement can be proved applying Weierstrass $M$ test).
For $y=0$ there is a jump in the cumulative distribution function of amplitude $P(0,t)$.}
\end{remark}
\begin{example}[The normal compound Poisson process]
{\em As an example of compound Poisson process, consider the case in which $X_i \sim \mathcal{N}(\mu,\sigma^2)$, so that
the cumulative distribution function is 
\begin{equation}
\label{normalcdf}
F_X(x) = \Phi(x|\mu,\sigma^2) = \frac{1}{2} \left( 1 + \mathrm{erf} \left( \frac{x - \mu}{\sqrt{2 \sigma^2}} \right) \right),
\end{equation}
where
\begin{equation}
\label{errorfunction}
\mathrm{erf} (x) = \frac{1}{\sqrt{\pi}} \int_0^x \mathrm{e}^{-u^2} \, du
\end{equation}
is the error function. In this case, the convolution $F_{Y_n}^{*n} (y)$ is given by $\Phi(y|n\mu,n\sigma^2)$ and one finds
\begin{equation}
\label{NCPPcdf}
F_{Y(t)}(y,t) = \mathrm{e}^{-\lambda t} \sum_{n=0}^{\infty} \frac{(\lambda t)^n}{n!} \Phi(y|n\mu,n\sigma^2).
\end{equation}}
\end{example}
\begin{corollary}
\label{corfundtheorgen}
In the same hypotheses as in Theorem \ref{fundtheorgen}, the probability density $f_{Y(t)}(y,t)$ of the process $Y(t)$ is given by
\begin{equation}
\label{probdens}
f_{Y(t)} (y,t) = P(0,t) \, \delta(y) + \sum_{n=1}^{\infty} P(n,t) f^{*n}_{Y_n} (y), 
\end{equation}
where $f^{*n}_{Y_n} (y)$ is the $n$-fold convolution of the probability density function $f_{Y_n} (y) = d F_{Y_n} (y) / dy$.
\end{corollary}
\begin{proof}
One has that $f_{Y(t)} (y,t) = d F_{Y(t)} (y,t)/ dy$; moreover, equation (\ref{probdens}) is the formal term-by-term derivative of equation
(\ref{cumulativecompoundrenewal}). If $y \neq 0$, there is no singular term and the series converges uniformly ($f^{*n}_{Y_n} (y)$
is bounded and Weierstrass $M$ test applies), therefore, for any $y$ it converges to the derivative of $F_{Y(t)}(y,t)$. This is true
also for $y=0$ for $n \geq 1$ and the jump in $y=0$ gives the singular term of weight $P(0,t)$ (see equation (\ref{alternatecum})).
\end{proof}
\begin{remark}[Historical news and applications]
{\em The distribution in equation (\ref{cumulativecompoundrenewal}) is also known as {\em generalized Poisson law}. This class of
distributions was studied by W. Feller in a famous work published in 1943 \cite{feller43}. It is useful to quote an excerpt from
Feller's paper, with notation adapted to the present paper.
\begin{quote}
The most frequently encountered application of the generalized Poisson distribution is to problems of the following type.
Consider independent random events for which the simple Poisson distribution may be assumed, such as: telephone calls,
the occurrence of claims in an insurance company, fire accidents, sickness, and the like. With each event there may be associated
a random variable $X$. Thus, in the above examples, $X$ may represent the length of the ensuing conversation, the sum under risk,
the damage, the cost (or length) of hospitalization, respectively. To mention an interesting example of a different type,
A. Einstein Jr. \cite{einstein} and G. Polya \cite{polya1,polya2} have studied a problem arising out of engineering practice
connected with building of dams, where the events consists of the motions of a stone at the bottom of a river; the variable
$X$ is the distance through which the stone moves down the river.

Now, if $F(x)$ is the cumulative distribution function of the variable $X$ associated with a single event, {\em then} 
$F^{*n} (x)$ is the cumulative distribution function of the accumulated variable associated with $n$ events. Hence the following equation
$$
G(x) = \mathrm{e}^{-a} \sum_{n=0}^{\infty} \frac{a^n}{n!} F^{*n} (x) 
$$
is the probability law of the sum of the variables (sum of the conversation times, total sum paid by the company, total damage, total
distance travelled by the stone, etc.).

In view of the above examples, it is not surprising that the law, or special cases of it, have been discovered, by various means 
and sometimes under disguised forms, by many authors.
\end{quote}
Indeed, the rediscovery and/or reinterpretation of equation (\ref{cumulativecompoundrenewal}) went on also after Feller's paper.
In physics, $X$ is interpreted as the position of a walker on a lattice and $N(t)$ is the number of walker jumps occured up to time $t$
\cite{montroll,scher,shlesinger,metz,klaft}. In finance, $X$ is the tick-by-tick log-return for a stock and $N(t)$ is the number of 
transactions up to time $t$ \cite{scalas06}.}
\end{remark}
The application of Fourier and Laplace transforms to equation (\ref{probdens}) leads to an equation which is known as Montroll-Weiss equation
in the physics literature \cite{montroll}. For reasons which will become clear in the following, it can also be called semi-Markov
renewal equation. Let 
$$
\widehat{f} (\kappa) = \int_{-\infty}^{+\infty} \mathrm{e}^{i \kappa x} f(x) \, dx
$$
denote the Fourier transform of $f(x)$ and
$$
\widetilde{g} (s) = \int_{0}^{\infty} \mathrm{e}^{-st} g(t) \, dt
$$
denote the Laplace transform of $g(t)$, then the following theorem holds true.
\begin{theorem}(Montroll-Weiss equation)
\label{MW}
Let $J$ denote the sojourn time of the Poisson process in $N(t) = 0$, with $F_J (t) = 1 - \mathrm{e}^{-\lambda t}$,
$f_J (t) = \lambda \mathrm{e}^{-\lambda t}$ and $P(0,t) = 1-F_J(t) = \mathrm{e}^{-\lambda t}$. We have that:
\begin{enumerate}
\item The Fourier-Laplace transform of the probability density $f_{Y(t)} (y,t)$ is given by
\begin{equation}
\label{MWI}
\widetilde{\widehat f}_{Y(t)}(\kappa,s) = \frac{1-\widetilde{f}_J (s)}{s} \frac{1}{1-\widetilde{f}_J (s) \widehat{f}_X (\kappa)}.
\end{equation}
\item The probability density function $f_{Y(t)} (y,t)$ obeys the following integral equation
\begin{equation}
\label{MWII}
f_{Y(t)} (y,t) = P(0,t) \delta(y) + \int_0^t \, dt' f_J (t-t') \int_{-\infty}^{+\infty} dy' f_{X} (y-y') f_{Y(t)} (y',t').
\end{equation}
\end{enumerate}
\end{theorem}
In order to prove Theorem (\ref{MW}), we need the following lemma.
\begin{lemma}
Let $T_1, T_2, \ldots, T_n, \ldots$ denote the epoch of the first, the second, $\ldots$, the $n$-th, $\ldots$ event of a Poisson process, respectively.
Let $J = J_1 = T_1$ denote the initial sojourn time and, in general, let $J_i = T_i - T_{i-1}$ be the $i$-th sojourn time. Then
$\{J_i \}_{i=1}^\infty$ is a sequence of i.i.d. random variables.
\end{lemma}
\begin{proof}
The proof of this lemma can be derived combining lemma 2.1 and proposition 2.12 in the book by Cont and Tankov \cite{cont}.
\end{proof}
It is now possible to prove the theorem
\begin{proof} (Theorem \ref{MW})
Let us start from equation (\ref{probdens}) and compute its Fourier-Laplace transform. It is given by
\begin{equation}
\label{MWIbis}
\widetilde{\widehat f}_{Y(t)}(\kappa,s) = \widetilde{P}(0,s)  + \sum_{n=1}^{\infty} \widetilde{P} (n,s) 
\left[ \widehat{f}_X (\kappa) \right]^n.
\end{equation}
Now, we have that
\begin{equation}
\label{possum}
T_n = \sum_{i=1}^n J_i
\end{equation}
is a sum of i.i.d. positive random variables and $P(n,t) \stackrel{\text{def}}{=}  \mathbb{P}(N(t) = n)$ meaning that
there are $n$ jumps up to $t= t_n$ and no jumps in $t-t_n$. Therefore, from pure probabilistic considerations, one has that
\begin{equation}
P(n,t) = P(0,t-t_n) * f_{T_n} (t_n)
\end{equation}
and, as a consequence of equation (\ref{possum}), one further has that 
\begin{equation}
f_{T_n} (t_n) = f_J^{*n} (t_n).
\end{equation}
Therefore, one can conclude that
\begin{equation}
\label{LaplaceCounting}
\widetilde{P} (n,s) = \widetilde{P}(0,s) \left[ \widetilde{f}_J (s) \right]^n.
\end{equation}
Inserting this result in equation (\ref{MWIbis}), noting that $\widetilde{P} (0,s) =(1 - \widetilde{f}_{J} (s)) /s$,
and summing the geometric series ($f_X$ and $f_J$ are probability densities) leads to equation (\ref{MWI}):
\begin{eqnarray}
\widetilde{\widehat f}_{Y(t)}(\kappa,s) & = & \widetilde{P}(0,s) +\widetilde{P}(0,s) 
\sum_{n=1}^{\infty} \left[\widetilde{f}_J (s) \widehat{f}_X (\kappa) \right]^n \nonumber \\ 
& = & \widetilde{P}(0,s) \sum_{n=0}^{\infty}  \left[\widetilde{f}_J (s) \widehat{f}_X (\kappa) \right]^n \nonumber \\
& = & \frac{1-\widetilde{f}_J (s)}{s} \frac{1}{1-\widetilde{f}_J (s) \widehat{f}_X (\kappa)}.
\end{eqnarray} 
Equation (\ref{MWI}) can be re-written as
\begin{equation}
\widetilde{\widehat f}_{Y(t)}(\kappa,s) = \frac{1-\widetilde{f}_J (s)}{s} + \widetilde{f}_J (s) \widehat{f}_X (\kappa)
\widetilde{\widehat f}_{Y(t)}(\kappa,s);
\end{equation}
Fourier-Laplace inverting and recalling the behaviour of convolutions under Fourier-Laplace transform, immediately 
leads to equation (\ref{MWII}).
\end{proof}
\begin{remark}{\em
Theorem (\ref{MW}) was proved in the hypothesis that $N(t)$ is a Poisson process. In this case, one
has $\widetilde{P}(0,s) = 1/(\lambda + s)$ and $\widetilde{f}_J (s) = \lambda/(\lambda+s)$ and
equation (\ref{MWI}) becomes
\begin{equation}
\widetilde{\widehat{f}}_{Y(t)}(\kappa,s) = \frac{1}{\lambda - \lambda \widehat{f}_X (\kappa) +s}.
\end{equation}
The inversion of the Laplace transform yields the characteristic function of the compound Poisson process
\begin{equation}
\label{charcomppoi}
\widehat{f}_{Y(t)} (\kappa, t) = \mathbb{E}\left( e^{i \kappa Y(t)} \right) = \mathrm{e}^{-\lambda(1 - \widehat{f}_X (\kappa))t}.
\end{equation}}
\end{remark}
\begin{remark}{\em
The proof of Theorem (\ref{MW}) does not depend on the specific form of $P(0,t)$ and $f_J (t)$, provided that the
positive random variables $\{J\}_{i=1}^\infty$ are i.i.d.. Therefore, equations (\ref{MWI}) and (\ref{MWII}) are true
also in the case of general compound renewal processes starting from $Y(0) = 0$ at time $0$. }
\end{remark}

\section{Compound fractional Poisson processes}\label{cfpp}

\begin{definition}[Renewal process]
Let $\{ J \}_{i=1}^{\infty}$ be a sequence of i.i.d. positive random variables interpreted as sojourn times between
subsequent events arriving at random time. They define a renewal process whose epochs of renewal (time instants at which the
events take place) are the random times
$\{ T \}_{n=0}^{\infty}$ defined by
\begin{eqnarray}
\label{epochs}
T_0 & = & 0, \nonumber \\
T_n & = & \sum_{i=1}^n J_i.
\end{eqnarray}
\end{definition}
The name renewal process is due to the fact that at any epoch of renewal, the process starts again from the beginning.
\begin{definition}[Counting process]
Associated to any renewal process, there is a counting process $N(t)$ defined as
\begin{equation}
\label{countdef}
N(t) = \max \{n : \, T_n \leq t \}
\end{equation}
that counts the number of events up to time $t$.
\end{definition}
\begin{remark}{\em 
As mentioned in the previous section $N(t)$ is the Poisson process if and only if $J \sim \exp(\lambda)$. Incidentally,
this is the only case of L\'evy and Markov counting process related to a renewal process (see 
\c{C}inlar's book \cite{cinlar75} for a proof of this statement). }
\end{remark}
\begin{remark}{\em
In this paper, we shall assume that the counting process has c\`adl\`ag sample paths. This means that the realizations are represented
by step functions. If $t_k$ is the epoch of the $k$-th jump, we have $N(t_k^-) = k-1$ and
$N(t_k^+) = k$.}
\end{remark}
Let $\{X_i \}_{i=1}^\infty$ be a sequence of independent and identically distributed (i.i.d.) real-valued random variables and let
$N(t)$, $t \geq 0$ denote the counting process. Further assume that the i.i.d. sequence and the counting process are independent.
We have the following
\begin{definition}[Compound renewal process]
The stochastic process
\begin{equation}
\label{compoundrenewalproc}
Y(t) = \sum_{i=1}^{N(t)} X_i
\end{equation}
is called compound renewal process.
\end{definition}
\begin{remark}{\em
Again, here, it is assumed that the sample paths are represented by c\`adl\`ag step functions. Compound renewal
processes generalize compound Poisson processes and they are called {\em continuous-time random walks}
in the physical literature.}
\end{remark}
\begin{remark}{\em
As compound renewal processes are just Markov chains (actually, random walks) subordinated to a counting process, their
existence can be proved as a consequence of the existence of the corresponding discrete-time random walks and counting
processes.}
\end{remark}
In general, compound renewal processes are non-Markovian, but they belong to the wider class of semi-Markov
processes \cite{cinlar75,flomenbom05,flomenbom07,janssen07,germano09}.
\begin{definition}[Markov renewal process]
A Markov renewal process is a two-component Markov chain $\{Y_n, T_n\}_{n=0}^{\infty}$,
where $Y_n$, $n \geq 0$ is a Markov chain and $T_n$, $n \geq 0$ is the $n$-th epoch
of a renewal process, homogeneous with respect to the second component and with
transition probability defined by 
\begin{equation}
\label{semimarkov1}
\mathbb{P}(Y_{n+1} \in A, J_{n+1} \leq t| Y_0, \ldots Y_{n}, J_1, \ldots, J_n) = 
\mathbb{P}(Y_{n+1} \in A, J_{n+1} \leq t| Y_{n}),
\end{equation}
where $A \subset \mathbb{R}$ is a Borel set and, as usual, $J_{n+1} = T_{n+1} - T_n$.
\end{definition}
\begin{remark}{\em
In this paper, homogeneity with respect to the first component will be assumed as well. 
Namely, if $Y_{n} = x$, the probability on the right-hand side of
equation (\ref{semimarkov1}) does not explicitly depend on $n$.}
\end{remark}
\begin{remark}[Semi-Markov kernel]{\em
The positive function $Q(x,A,t) = \mathbb{P}(Y_{n+1} = y \in A, J_{n+1} \leq t| Y_{n} =x)$, 
with $x \in \mathbb{R}$, $A \subset \mathbb{R}$ a Borel set,
and $t \geq 0$ is called semi-Markov kernel.}
\end{remark}
\begin{definition}[Semi-Markov process]
Let $N(t)$ denote the counting process defined as in equation (\ref{countdef}), the stochastic
process $Y(t)$ defined as
\begin{equation}
Y(t) = Y_{N(t)}
\end{equation}
is the semi-Markov process associated to the Markov renewal process $Y_n,T_n$, $n \geq 0$.
\end{definition}
\begin{remark}{\em
In equation (\ref{countdef}), $\max$ is used instead of the more general $\sup$ as 
only processes with finite (but arbitrary) number of jumps in $(0,t]$ are considered here. }
\end{remark}
\begin{theorem}
Compound renewal processes are semi-Markov processes with semi-Markov kernel given by
\begin{equation}
\label{crpkernel}
Q(x,A,t) = P(x,A) F_J (t),
\end{equation}
where $P(x,A)$ is the Markov kernel (a.k.a. Markov transition function or transition
probability kernel) of the random walk
\begin{equation}
P(x, A) \stackrel{\text{def}}{=} \mathbb{P}(Y_{n+1} \in A| Y_{n} = x),
\end{equation}
and $F_J (t)$ is the probability distribution function of sojourn times. Moreover,
let $f_X (x)$ denote the probability density function of jumps, one has
\begin{equation}
\label{rwkernel}
P(x, A) = \int_{A-x} f_J (u) \, du,
\end{equation}
where $A -x$ is the set of values in A translated of $x$ towards left.
\end{theorem}
\begin{proof}
The compound renewal process is a semi-Markov process by construction, where 
the couple $Y_n, T_n$, $n \geq 0$ defining the corresponding Markov renewal process
is made up of a random walk $Y_n$, $n \geq 0$ with $Y_0 = 0$ and a renewal process 
with epochs given by $T_n$, $n \geq 0$ with $T_0 = 0$. Equation (\ref{crpkernel})
is an immediate consequence of the independence between the random walk and
the renewal process. Finally, equation (\ref{rwkernel}) is the standard 
Markov kernel of a random walk whose jumps are i.i.d. random variables
with probability density function $f_X (x)$. 
\end{proof}
\begin{remark}{\em 
As a direct consequence of the previous theorem, if the law of the couple $X_n, J_n$ has a joint probability density function 
$f_{X,J}(x,t) = f_X(x) f_J (t)$, then one has
\begin{eqnarray}
\label{crpkernelbis}
\mathbb{P}(Y_{n+1} \in A, J_{n+1} \leq t| Y_n) & = & Q(x,A,t) = P(x,A) F_J(t) =  \nonumber \\
& = & \int_{A} f_X(u) \, du \int_{0}^{t}f_J (v)\,dv. 
\end{eqnarray}}
\end{remark}
\begin{theorem}(Semi-Markov renewal equation) The probability density function $f_{Y(t)} (y,t)$ of a compound
renewal process obeys the semi-Markov renewal equation
\begin{equation}
\label{semimarkovrenewal}
f_{Y(t)} (y,t) = P(0,t) \delta(y) + \int_0^t \, dt' f_J (t-t') \int_{-\infty}^{+\infty} dy' f_{X} (y-y') f_{Y(t)} (y',t').
\end{equation}
\end{theorem}
\begin{proof}
By definition, one has that
\begin{equation}
\label{smproof1}
\mathbb{P}(Y(t) \in dy| Y(0) = 0) = f_{Y(t)}(y,t) dy,
\end{equation}
and that
\begin{eqnarray}
\label{smproof2}
\mathbb{P}(Y(t) \in dy| Y(t') = y') & = & \mathbb{P}(Y(t-t') \in dy| Y(0) = y') = \nonumber \\
\mathbb{P}(Y(t-t') - y' \in dy| Y(0) = 0) & = & f_{Y(t)}(y-y',t-t') dy,
\end{eqnarray}
because the increments in time and space are i.i.d. and hence
homogeneous. From equation \eqref{crpkernelbis}, one further has
\begin{equation}
\label{smproof3}
\mathbb{P}(Y_{1} \in dy, J_{1} \in dt| Y_0=0) = f_X(y) f_J(t) dy dt.
\end{equation}
Now, the probability in equation \eqref{smproof1} can be decomposed into two mutually exclusive parts,
depending on the behaviour of the first interval
\begin{eqnarray}
\label{smproof4}
\mathbb{P}(Y(t) \in dy| Y(0) = 0) & = & \nonumber \\ 
\mathbb{P}(Y(t) \in dy, J_1 > t| Y(0) = 0) & +  & \mathbb{P}(Y(t) \in dy, J_1 \leq t| Y(0) = 0).
\end{eqnarray}
The part with no jumps up to time $t$ immediately gives
\begin{equation}
\label{smproof5}
\mathbb{P}(Y(t) \in dy, J_1 > t| Y(0) = 0) = P(0,t) \delta(y) dy,
\end{equation}
whereas the part with jumps becomes
\begin{multline}
\label{smproof6}
\mathbb{P}(Y(t) \in dy, J_1 \leq t| Y(0) = 0) = \\
\int_{-\infty}^{+\infty} \int_{0}^{t}
\mathbb{P}(Y(t) \in dy|Y(t') = y') \mathbb{P}(Y_1 \in dy',J_1 \in dt') = \\
\int_{-\infty}^{+\infty} \int_{0}^{t} f_{Y(t)}(y-y',t-t') dy f_X(y') f_J(t') dy' dt' = \\
\left[ \int_{-\infty}^{+\infty} \int_{0}^{t} f_{Y(t)}(y-y',t-t') f_X(y') f_J(t') dy' dt' \right] dy
\end{multline}
as a consequence of Bayes' formula and of equations \eqref{smproof2} and \eqref{smproof3}. A replacement
of equations \eqref{smproof1}, \eqref{smproof5}, \eqref{smproof6} into equation \eqref{smproof4} and a 
rearrangement of the convolution variables straightforwardly lead to the thesis \eqref{semimarkovrenewal}. 
\end{proof}
\begin{remark}{\em
Note that the semi-Markov renewal equation \eqref{semimarkovrenewal} does coincide with the Montroll-Weiss
equation \eqref{MWII} as anticipated.}
\end{remark}
\begin{definition}[Mittag-Leffler renewal process]
The sequence $\{ J_i \}_{i=1}^\infty$ of positive independent and identically distributed random
variables whose complementary cumulative distribution function $P_\beta (0,t)$ is given by
\begin{equation}
\label{MLsurvival}
P_\beta (0,t) = E_{\beta} ( - t^{\beta})
\end{equation}
defines the so-called Mittag-Leffler renewal process.
\end{definition}
\begin{remark}{\em
The one-parameter Mittag-Leffler function in \eqref{MLsurvival} is a straightforward generalization of the
exponential function. It is given by the following series
\begin{equation}
\label{MLfunction}
E_\beta (z) = \sum_{n=0}^{\infty} \frac{z^n}{\Gamma(\beta n +1)},
\end{equation}
where $\Gamma(z)$ is Euler's Gamma function. The Mittag-Leffler function coincides with the exponential
function for $\beta =1$. The function $E_{\beta}(-t^{\beta})$ is completely monotonic and it is $1$ for
$t=0$. Therefore, it is a legitimate survival function.  }
\end{remark}
\begin{remark}{\em 
The function $E_\beta (-t^\beta)$ is approximated by a stretched exponential for $t \to 0$:
\begin{equation}
\label{MLlowt}
E_\beta (-t^\beta) \simeq 1 - \frac{t^\beta}{\Gamma(\beta+1)} \simeq \mathrm{e}^{-t^\beta/\Gamma(\beta+1)},\;\; 
\mathrm{for}\,\, 0 < t \ll 1,
\end{equation}
and by a power-law for $t \to \infty$:
\begin{equation}
\label{MLlarget}
E_\beta (-t^\beta) \simeq \frac{\sin(\beta \pi)}{\pi} \frac{\Gamma(\beta)}{t^\beta},\;\; 
\mathrm{for}\,\, t \gg 1.
\end{equation}}
\end{remark}
\begin{remark}{\em 
For applications, it is often convenient to include a scale factor in the definition
\eqref{MLsurvival}, so that one can write
\begin{equation}
\label{MLsurvivalScaled}
P(0,t) = E_{\beta} \left( - ( t/\gamma_t )^{\beta} \right).
\end{equation}
As the scale factor can be introduced in different ways, the reader is warned to pay attention to
its definition. The assumption $\gamma_t =1$ made in \eqref{MLsurvival} is equivalent to a change of time unit.}
\end{remark}
\begin{theorem} (Mittag-Leffler counting process - fractional Poisson process) 
\label{CFPPTheorem}
The counting process $N_\beta (t)$ associated to the
renewal process defined by equation \eqref{MLsurvival} has the following distribution
\begin{equation}
\label{MLcounting}
P_{\beta} (n,t) = \mathbb{P} (N_\beta (t) = n) = \frac{t^{\beta n}}{n!} E_{\beta}^{(n)} (-t^\beta),
\end{equation}
where $E_{\beta}^{(n)} (-t^\beta)$ denotes the $n$-th derivative of $E_{\beta} (z)$ evaluated
at the point $z = -t^\beta$.
\end{theorem}
\begin{proof}
The Laplace transform of $P(0,t)$ is given by \cite{podlubny}
\begin{equation}
\label{MLLaplace1}
\widetilde{P}_\beta (0,s) = \frac{s^{\beta-1}}{1+s^\beta},
\end{equation}
as a consequence, the Laplace transform of the probability density function $f_{J,\beta} (t) = - d P_{\beta} (0,t) / dt$ is given by
\begin{equation}
\label{MLLaplace2}
\widetilde{f}_{J,\beta} (s) = \frac{1}{1+s^\beta};
\end{equation}
recalling equation \eqref{LaplaceCounting}, one immediately has
\begin{equation}
\label{LaplaceFracCounting}
\widetilde{P}_{\beta} (n,s) = \frac{1}{(1+s^\beta)^n} \frac{s^{\beta-1}}{1+s^\beta}.
\end{equation}
Using equation (1.80) in Podlubny's book \cite{podlubny} for the inversion of the Laplace transform
in \eqref{LaplaceFracCounting}, one gets the thesis \eqref{MLcounting}.
\end{proof}
\begin{remark}{\em 
The previous theorem was proved by Scalas {\em et al.} \cite{scalas04,mainardi04}. Notice that $N_1(t)$ is the
Poisson process with parameter $\lambda =1$. Recently, Meerschaert {\em et al.} 
\cite{meerschaert10} proved that the fractional Poisson process $N_\beta (t)$ coincides with the process defined
by $N_1 (E(t))$ where $E(t)$ is the functional inverse of the standard $\beta$-stable subordinator. The latter
process was also known as fractal time Poisson process. This result unifies different approaches to fractional
calculus \cite{beghin,meerschaert09}.}
\end{remark}
\begin{remark}{\em 
\label{markovlevy}
For $0 < \beta < 1$, the fractional Poisson process is semi-Markov, but not Markovian and is not L\'evy. The process
$N_\beta (t)$ is not Markovian as the only Markovian counting process is the Poisson process \cite{cinlar75}. It is not
L\'evy as its distribution is not infinitely divisible.   }
\end{remark}
\begin{definition}[Compound fractional Poisson process]
With the usual hypotheses, the process
\begin{equation}
\label{CFPP}
Y_\beta (t) = Y_{N_{\beta} (t)} = \sum_{i=1}^{N_{\beta}(t)} X_i 
\end{equation}
is called compound fractional Poisson process.
\end{definition}
\begin{remark}{\em 
The process $Y_1 (t)$ coincides with the compound Poisson process of parameter $\lambda =1$.}
\end{remark}
\begin{theorem}
Let $Y_\beta (t)$ be a compound fractional Poisson process, then
\begin{enumerate}
\item its cumulative distribution function $F_{Y_\beta (t)} (y,t)$ is given by
\begin{equation}
\label{CFPPcdf}
F_{Y_\beta (t)} (y,t) = E_\beta (-t^\beta) \theta (y) + \sum_{n=1}^\infty \frac{t^{\beta n}}{n!} E_{\beta}^{(n)} (-t^\beta) F_{Y_n}^{*n} (y);
\end{equation}
\item its probability density $f_{Y_\beta (t)} (y,t) $ function is given by
\begin{equation}
\label{CFPPpdf}
f_{Y_\beta (t)} (y,t) = E_\beta (-t^\beta) \delta (y) + \sum_{n=1}^\infty \frac{t^{\beta n}}{n!} E_{\beta}^{(n)} (-t^\beta) f_{Y_n}^{*n} (y);
\end{equation}
\item its characteristic function $\widehat{f}_{Y_\beta (t)} (\kappa,t)$ is given by
\begin{equation}
\label{CFPPcf}
\widehat{f}_{Y_\beta (t)} (\kappa,t) = E_\beta \left[ t^\beta(\widehat{f}_X(\kappa)-1) \right].
\end{equation}
\end{enumerate}
\end{theorem}
\begin{proof}
The first two equations \eqref{CFPPcdf} and \eqref{CFPPpdf} are a straightforward consequence of Theorem 
\ref{fundtheorgen}, Corollary \ref{corfundtheorgen} and Theorem \ref{CFPPTheorem}. Equation \eqref{CFPPcf} is
the straightforward Fourier transform of \eqref{CFPPpdf}.
\end{proof}
\begin{remark}{\em 
For $0< \beta < 1$, the compound fractional Poisson process is not Markovian and not L\'evy (see Remark \ref{markovlevy}).}
\end{remark}

\section{Limit theorems}\label{lt}
\begin{definition}[Space-time fractional diffusion equation] Let $\partial^\alpha / \partial |x|^{\alpha}$ denote the spatial non-local
pseudo-differential operator whose Fourier transform is given by 
\begin{equation}
\label{rieszfeller}
\mathcal{F} \left[ \frac{\partial^\alpha f(x)}{\partial |x|^{\alpha}}; \kappa \right]= -|\kappa|^{\alpha} \widehat{f} (\kappa),
\end{equation}
for $x \in (-\infty,+\infty)$, $0 < \alpha \leq 2$ and let $\partial^\beta / \partial t^\beta$
denote the time non-local pseudo-differential operator whose Laplace transform is given by
\begin{equation}
\label{caputo}
\mathcal{L} \left[ \frac{\partial^\beta g(t)}{\partial t^\beta} ; s \right] = s^{\beta} \widetilde{g} (s) - s^{\beta-1} g(0^+),
\end{equation}
for $t > 0$, $0< \beta \leq 1$. Then the pseudo-differential equation
\begin{equation}
\label{fractionaldiffusion}
\frac{\partial^\alpha u(x,t)}{\partial |x|^{\alpha}} = \frac{\partial^\beta u(x,t)}{\partial t^\beta}
\end{equation}
is called space-time fractional differential equation.
\end{definition}
\begin{remark}{\em 
The operator $\partial^\alpha / \partial |x|^{\alpha}$ is called Riesz derivative and is discussed by
Saichev and Zaslavsky \cite{saichev}. The operator $\partial^\beta / \partial t^\beta$ is called Caputo
derivative and was introduced by Caputo in 1967 \cite{caputo} as a regularization of the so-called Riemann-Liouville
derivative.}
\end{remark}
\begin{theorem}(Cauchy problem for the space-time fractional diffusion equation) Consider the following 
Cauchy problem for the space-time fractional diffusion equation \eqref{fractionaldiffusion}
\begin{eqnarray}
\label{cauchy}
\frac{\partial^\alpha u_{\alpha,\beta} (x,t)}{\partial |x|^{\alpha}} & = & \frac{\partial^\beta u_{\alpha,\beta} (x,t)}{\partial t^\beta} \nonumber \\
u_{\alpha,\beta} (x,0^+) & = & \delta(x),
\end{eqnarray}
then the function
\begin{equation}
\label{greenfunction}
u_{\alpha,\beta} (x,t) = \frac{1}{t^{\beta/\alpha}} W_{\alpha,\beta} \left( \frac{x}{t^{\beta/\alpha}} \right),
\end{equation}
where
\begin{equation}
\label{scalinggreen}
W_{\alpha,\beta} (u) = \frac{1}{2 \pi} \int_{-\infty}^{+\infty} d \kappa \, \mathrm{e}^{-i \kappa u} E_{\beta} (-|\kappa|^\alpha),
\end{equation}
solves the Cauchy problem \cite{lumapa}.
\end{theorem}
\begin{proof} Taking into account the initial condition \eqref{cauchy}, as a consequence of the operator definition, for non-vanishing
$\kappa$ and $s$,
the Fourier-laplace transform of equation \eqref{fractionaldiffusion} is given by
\begin{equation}
-|\kappa|^\alpha \widehat{\widetilde{u}}_{\alpha,\beta} (\kappa, s) = s^\beta \widehat{\widetilde{u}}_{\alpha,\beta} (\kappa,s) - s^{\beta-1},
\end{equation}
leading to
\begin{equation}
\widehat{\widetilde{u}}_{\alpha,\beta} (\kappa, s) = \frac{s^{\beta - 1}}{|\kappa|^\alpha + s^\beta}.
\end{equation}
Equation \eqref{MLLaplace1} can be invoked for the inversion of the Laplace transform yielding
\begin{equation}
\label{GFFourier}
\widehat{u}_{\alpha,\beta} (\kappa, t) = E_\beta (-t^\beta |\kappa|^\alpha).
\end{equation}
Eventually, the inversion of the Fourier transform leads to the thesis.
\end{proof}
\begin{remark}{\em 
The function defined by equations \eqref{greenfunction} and \eqref{scalinggreen} is a probability density function. For $\beta =1$ and
$\alpha = 2$, it coincides with the Green function for the ordinary diffusion equation. The case $\beta = 1$ and $0 < \alpha \leq 2$ gives
the Green function and the transition probability density for the symmetric and isotropic $\alpha$-stable L\'evy process
$L_\alpha (t)$ \cite{jacob}.}
\end{remark}
\begin{theorem}
Let $Y_{\alpha,\beta} (t)$ be a compound fractional Poisson process and let $h$ and $r$ be two scaling factors such that
\begin{eqnarray}
Y_n (h) & = & h X_1 + \ldots + h X_n \\
T_n (r) & = & r J_1 + \ldots + r J_n,
\end{eqnarray}
and
\begin{equation}
\lim_{h,r \to 0} \frac{h^\alpha}{r^\beta} = 1,
\end{equation}
with $0< \alpha \leq 2$ and $0 < \beta \leq 1$. To clarify the role of the parameter $\alpha$, further assume that,
for $h \to 0$, one has
\begin{equation}
\label{spacescaling}
\widehat{f}_X (h \kappa) \sim 1 - h^{\alpha}|\kappa|^{\alpha},
\end{equation}
then, for $h,r \to 0$ with $h^\alpha / r^\beta \to 1$, $f_{hY_{\alpha,\beta} (rt)} (x,t)$ weakly converges to $u_{\alpha,\beta} (x,t)$, the Green
function of the fractional diffusion equation. 
\end{theorem}
\begin{proof}
In order to prove weak convergence, it suffices to show the convergence of the characteristic function \eqref{CFPPcf} \cite{feller}. Indeed,
one has
\begin{equation}
\label{convergence}
\widehat{f}_{h Y_{\alpha,\beta} (rt)} (\kappa, t) = E_\beta \left( -\frac{t^\beta}{r^{\beta}} (\widehat{f}_X (h \kappa ) - 1) \right)  
\stackrel{h,r \to 0}{\to}  E_\beta (-t^\beta |\kappa|^\alpha),
\end{equation}
which completes the proof.
\end{proof}
\begin{remark}{\em 
Condition \eqref{spacescaling} is not void. It is satisfied by all the distributions belonging to the basin of attaction of
symmetric $\alpha$-stable laws. Let $f_{\alpha,X} (x)$ denote the probability density function of a
symmetric $\alpha$-stable law whose characteristic function is
\begin{equation}
\label{CFstable}
\widehat{f}_{\alpha,X} (\kappa) = \mathrm{e}^{-|\kappa|^\alpha},
\end{equation}
then one can immediately see that \eqref{spacescaling} holds true. As above, let $L_\alpha (t)$ denote the symmetric $\alpha$-stable
L\'evy process. Then, equation \eqref{GFFourier} is the characteristic function of $L_{\alpha,\beta}(t) = L_\alpha(N_\beta (t))$, that is
of the symmetric $\alpha$-stable L\'evy process subordinated to the fractional Poisson process. This remark leads to the conjecture that
$L_{\alpha,\beta} (t)$ is the functional limit of $Y_{\alpha,\beta} (t)$, the $\alpha$-stable compound fractional Poisson process
defined by equation \eqref{CFPP} with the law of the jumps $X$ belonging to the basin of attraction of or coinciding with an $\alpha$-stable
law. This conjecture can be found in a paper by Magdziarz and Weron \cite{magdziarz} and is proved in Meerschaert {\em et al.} 
\cite{meerschaert10} using the methods discussed in the book by Meerschaert and Scheffler \cite{meerschaert01}.}
\end{remark}

\section*{Acknowledgments}

Enrico Scalas is grateful
to Universitat Jaume I for the financial support received from their Research Promotion
Plan 2010 during his scientific visit in Castell\'on de la Plana where this paper 
was completed. In writing this paper, Enrico profited of useful discussions with
Rudolf Gorenflo, Francesco Mainardi, and Mark M. Meerschaert.


\end{document}